\documentclass[a4paper,12pt]{article}

\usepackage{graphicx}
\usepackage{caption,color}   

\usepackage{comment}
\usepackage{amssymb}
\usepackage{amsthm}
\usepackage{hyperref}
\usepackage{amsmath}
\usepackage{epic}
\usepackage{CJK}
\usepackage{setspace}
\usepackage{float}
\usepackage{multirow}
\usepackage{tikz}
\usepackage{booktabs}
\usepackage{multirow}
\usepackage{nicefrac}
\usepackage{longtable}
\newtheorem{thm}{Theorem}[section]
\newtheorem{cor}[thm]{Corollary}
\newtheorem{lem}[thm]{Lemma}
\newtheorem{pro}[thm]{Proposition}

\newtheorem{example}[thm]{Example}
\newtheorem{rem}[thm]{Remark}
\newtheorem{conjecture}[thm]{Conjecture}

\usepackage{xcolor}

\title{Spectral bounds for distance coloring \\and packing  parameters of graphs \\via semidefinite programming}

\author{Aida Abiad \thanks{Department of Mathematics and Computer Science, Eindhoven University of Technology, Eindhoven, The Netherlands (\texttt{a.abiad.monge@tue.nl})\\ Department of Mathematics and Data Science of Vrije Universiteit Brussel, Brussels, Belgium}
\quad
Yue Yang 
\thanks{College of Mathematical Sciences, Harbin Engineering University, Harbin 150001, PR China (\texttt{yangyueyy@hrbeu.edu.cn})}
\quad
Jiang Zhou 
\thanks{College of Mathematical Sciences, Harbin Engineering University, Harbin 150001, PR China (\texttt{zhoujiang@hrbeu.edu.cn})}
}

\date{}

\begin{document}
\maketitle

\begin{abstract}
Using methods from spectral graph theory and semidefinite programming, we obtain sharp spectral bounds for several graph parameters related to distance colorings and packing, including the injective chromatic number, the open packing number, the injective chromatic index, and the strong chromatic index. The new spectral bounds improve several existing combinatorial bounds. Furthermore, we apply the obtained eigenvalue bounds on the first two parameters to estimate the code covering number and the open packing number of hypercubes, obtaining new exact values and strengthened bounds regarding existing literature results. The obtained results illustrate the power of combining spectral and semidefinite programming tools for tackling coloring and packing problems in graph theory and coding theory.  \\


\end{abstract}

\section{Introduction}

The interplay between structural properties of graphs and the spectra of their associated matrices is a central theme in discrete mathematics, and one that has driven much of the development of spectral graph theory. Classical work in the area has focused on fundamental graph invariants such as the independence number and the chromatic number, following the seminal work of Hoffman and Lovász, among others. In recent years, however, more refined colouring and packing parameters have been introduced and investigated in the literature.

In this work we investigate the following graph parameters. Let $G=(V,E)$ be a graph with $|V|=n$ vertices and $|E|=m$ edges.
\begin{itemize}
\item In an \emph{injective coloring}, for every vertex $v \in V$, all the neighbors of $v$ are assigned with distinct colors. The \emph{injective chromatic number} $\chi_i(G)$ is the minimum number of colors needed for an injective coloring. 
\item An \emph{open packing} of $G$ is a vertex subset $S$ if any two vertices in $S$ do not have a common neighbor in $G$. The \emph{open packing number} $\rho^o(G)$ is the maximum cardinality among all open packings of $G$. 
\item An edge coloring of $G$ is an \emph{injective edge coloring} if for any two distinct edges $e_1$ and $e_2$, the colors of $e_1$ and $e_2$ are distinct if they are at distance $1$ in $G$ or in a common triangle. The \emph{injective chromatic index} $\chi_i'(G)$ is the minimum number of colors needed for an injective edge coloring of $G$.  
\item A \emph{strong edge coloring} of $G$ is an assignment of colors to the edges of $G$ such that two distinct edges are colored differently if their distance is at most $1$. The \emph{strong chromatic index} $\chi_s'(G)$ is the smallest number of colors in any strong edge coloring. 
\end{itemize}

These four parameters exhibit quite a different behavior from their classical counterparts. Thus, specific tools to estimate them have been proposed in the literature -- for a survey paper on graph coloring parameters, we refer the reader to  \cite{CranstonDS}. The above parameters have become important in modeling constraints arising in areas such as communication networks (where interference constraints extend to vertices at distance two), scheduling, steganography and coding theory (where distance constraints on codes translate into colouring and packing problems on associated graphs). These more refined parameters can be viewed as instances of coloring problems on graph powers or line graphs, modeling constraints that propagate beyond immediate adjacency.

Despite their relevance for multiple applications, the above parameters are computationally intractable (see e.g. \cite{01, Shalu, McLoughlin,complexityedgeinjective,PG2023}), making the development of sharp and easily computable bounds essential; this motivates the search for eigenvalue bounds. Existing bounds for quantities such as $\chi_i(G)$ and $\chi_s'(G)$ are typically expressed in terms of local parameters like the maximum degree $\Delta$, order, or size \cite{02,01,MR1997}, and therefore often fail to capture the global graph symmetries and regularity (which are encoded in the graph spectrum).

Motivated by the above, in this paper we derive eigenvalue bounds for the aforementioned parameters by combining techniques from spectral graph theory and positive semidefinite programming. In particular, we encode combinatorial distance constraints as semidefinite conditions, allowing us to optimise over suitably structured matrix families rather than only the adjacency spectrum. We also show that our bounds are tight and can outperform the following known results:
\begin{itemize}
\item Lower bound on the \emph{injective chromatic number} that can improve a known result by Bre\v{s}ar, Samadi and Yero \cite[Theorem 6]{02}. As an application to hypercubes of our spectral bound for the injective chromatic number (in particular, to estimate the code covering number in error correction), we obtain new sufficient and necessary conditions for its exact value, and new lower bounds, extending the results by Hahn, Kratochvı\'il, \v{S}ir\'a\v{n} and Sotteau \cite{01}.

\item  Upper bounds on the \emph{open packing number} of hypercubes that, when applied to hypercubes, improves upon
previous results by Bre\v{s}ar et al. \cite{BKR2024} and Brimkov et al. \cite[Corollary 3.12]{BCG2025}.
    
\item Extension of a lower bound on the \emph{injective chromatic index} by Kostochka et al. \cite[Proposition 8]{KRX2021} which holds for $d$-regular bipartite graphs to general graphs. For some graphs, our lower bounds are better than the bound of Cardoso et al. \cite[Proposition 3.1]{Cardosofilomat}.

\item Lower bound on the \emph{strong chromatic index} that can improve a known result by Abiad and Reijnders \cite[Corollary 9]{AR2026}.
\end{itemize}

This paper is organized as follows. We start, in Section \ref{sec:preliminaries}, presenting the needed definitions and some preliminary auxiliary lemmas. In Section \ref{sec:boundsinjectivechrom}, we focus on the injective chromatic number $\chi_i(G)$ and establish a spectral lower bound (Theorem \ref{thm:evinjectivelowerbound}). We show that this bound is sharp for some graph classes, derive an analogous version for direct products, and show that the regular version (Corollary \ref{coro:injectiveboundregular}) can improve upon \cite[Theorem 6]{02}. In Section \ref{sec:boundopenpackingnumber}, we derive spectral upper bounds for the open packing number $\rho^o(G)$, which we extend to direct products and show that are attainable. Applications of the new bound to bipartite prism graphs yield improved bounds for the open packing number of hypercubes $Q_n$. In Section \ref{section:Injectivechromaticindex} we obtain spectral lower bounds for the injective chromatic index $\chi_i'(G)$, present classes of graphs attaining them (Theorem \ref{ICIbound}), and show that we obtain improvements upon \cite[Proposition 3.1]{Cardosofilomat} in some cases. In Section \ref{section:Strongchromaticindex}, we establish lower bounds for the strong chromatic index $\chi_s'(G)$ and identify families for which these bounds are tight. In particular, Theorems \ref{SCIbound1} and \ref{SCIbound2} improve the bound in \cite[Corollary 9]{AR2026} for certain graphs. Finally, Section \ref{sec:apps} contains applications: in particular, in Section \ref{sec:codecoverignumber} we apply our spectral approach to the code covering number of hypercubes (Theorem \ref{thm:codecoveringspectralbound}), extending known results, and in Section \ref{sec:bipartiteprisms} we investigate open packing in bipartite prism graphs, yielding further improvements on known results on hypercubes (Proposition \ref{propo:openpackingbipartiteprisms}).

\section{Preliminaries}\label{sec:preliminaries}

Let $V(G)$ and $E(G)$ denote the vertex set and edge set of a finite simple graph $G$, respectively. The \emph{adjacency matrix} of $G$, denoted by $A_G$, is a real symmetric matrix indexed by $V(G)$ such that $(A_G)_{uv} = 1$ if $uv\in E(G)$ and $(A_G)_{uv}= 0$ otherwise. The open neighborhood of a vertex $v$ in $G$ is denoted by $N_G(v)=\{u\in V(G):uv\in E(G)\}$. The maximum and minimum degrees in $G$ are denoted by $\Delta=\Delta(G)$ and $\delta=\delta(G)$, respectively. A vertex coloring of $G$ is \emph{proper} if adjacent vertices are assigned distinct colors. The \textit{chromatic number} of $G$, denoted by $\chi(G)$, is the smallest number of colors used in a proper coloring of $G$. The following spectral bound on $\chi(G)$ is a seminal result by Hoffman.

\begin{lem}[Hoffman bound] \cite{Hoffman} \label{HoffmanChi}
Let $G$ be a graph with at least one edge, and let $\lambda_1$ and $\lambda_n$ be the maximum and minimum eigenvalues of $A_G$, respectively. Then
\begin{eqnarray*}
\chi(G)\geq1-\frac{\lambda_1}{\lambda_n}.
\end{eqnarray*}
\end{lem}

A coloring that meets the Hoffman bound above is called a \textit{Hoffman coloring}.

\begin{lem}\cite[Proposition 2.3]{Blokhuis}\label{BlokhuisProposition 2.3}
If $G$ is an regular graph with a Hoffman coloring, then all color classes of the Hoffman coloring have equal size.
\end{lem}

The \textit{independence number} of $G$, denoted by $\alpha(G)$, is the maximum cardinality of independent sets in $G$. The following result is another classical spectral bound for $\alpha(G)$, also due to Hoffman.

\begin{lem}[Ratio bound]\cite{{BrouwerHaemers}}\label{Haemersbound}
Let $G$ be an $n$-vertex graph with minimum degree $\delta$, and let $\lambda_1\geq\cdots\geq\lambda_n$ be eigenvalues of $A_G$. Then
\begin{eqnarray*}
\alpha(G)\leq n\frac{-\lambda_1\lambda_n}{\delta^2-\lambda_1\lambda_n}.
\end{eqnarray*}
\end{lem}

\begin{lem}\cite[Theorem 2.1.4]{Haemers}\label{subgraphbound}
Let $G$ be an $n$-vertex $d$-regular graph with adjacency eigenvalues $\lambda_1\geq\cdots\geq\lambda_n$, and let $G_1$ be an $n_1$-vertex $d_1$-regular induced subgraph of $G$. Then
\begin{eqnarray*}
nd_1-n_1d\geq(n-n_1)\lambda_n.
\end{eqnarray*}
\end{lem}

For a given graph $G$, the \textit{two-step graph}  (also called the \emph{neighboring graph}) of $G$ \cite{02}, denoted as $G^{(2)}$, is the graph having the same vertex set as $G$ with an edge joining two vertices in $G^{(2)}$ if and only if they have a common neighbor in $G$. Obviously then it holds 
\begin{equation}\label{eq:relationstrongandchromaticnumbers}
\chi_i(G)=\chi(G^{(2)})\leq \chi(G^2).  \end{equation}
The chromatic number of $G^2$ has important applications in steganography, see \cite{FL2007}.

The celebrated Lov\'asz theta function $\vartheta(G)$ (also called the Lov\'asz number), was introduced in \cite{26} as an upper bound on the Shannon capacity of a graph. The Lov\'asz theta function satisfies the following \emph{sandwich inequality}:
\begin{eqnarray*}
\omega(G)\leq \vartheta(\overline{G})\leq \chi(G),
\end{eqnarray*}
where $\overline{G}$ denotes the complement of the graph $G$, $\omega(G)$ and $\chi(G)$ are respectively the clique number and the chromatic number of $G$. The equality cases were studied in \cite{YangZhou,Zhou}.

Let $(M)_{ij}$ denote the $(i,j)$-entry of a matrix $M$. For a real symmetry matrix $M$, let $\lambda_i(M)$ denote the $i$-th largest eigenvalue of $M$.

Next we present two preliminary lemmas that will be used to derive most of our main results.
\begin{lem}\cite{31}\label{lema:balla}
Let $G$ be an $n$-vertex graph, and let $\mathcal{B}$ be the set of $n\times n$ positive semidefinite real matrices $B$ such that $(B)_{ii}= 1$ for all $i\in V(G)$ and $(B)_{ij}= 0$ whenever $ij\notin E(G)$. Then
\begin{eqnarray*}
\vartheta(\overline{G})=\max_{B\in\mathcal{B}}\lambda_{1}(B).
\end{eqnarray*}
\end{lem}

Given two graphs $G$ and $H$ with vertex sets $V(G)$ and $V(H)$, respectively, their \emph{direct product} $G\otimes H$ is the graph with vertex set $V(G)\times V(H)$, and two vertices $(u_1, v_1)$ and $(u_2,v_2)$ are adjacent in $G\otimes H$ if and only if $\{u_1,u_2\}\in E(G)$ and $\{v_1,v_2\}\in E(H)$. Next we show that the inequality in \cite[Theorem 1]{18} also holds for the direct product of general graphs.

\begin{lem}\label{lema:injectivekronecker}
For any two graphs $G$ and $H$, we have
\begin{eqnarray*}
\chi_i(G\otimes H)\leq \chi_i(G)\chi_i(H).
\end{eqnarray*}
\end{lem}

\begin{proof}
There exist partitions $V(G)=S_1\cup\cdots\cup S_k$ ($k=\chi_i(G)$) and $V(H)=T_1\cup\cdots\cup T_m$ ($m=\chi_i(H)$) such that $S_1,\ldots,S_k$ are open packings of $G$ and $T_1,\ldots,T_m$ are open packings of $H$. Since $$V(G\otimes H)=\bigcup\limits_{\substack{1\leq i \leq k \\ 1\leq j\leq m}} S_i\times T_j$$ and $S_i\times T_j$ is an open packing of $G \otimes H$ for any $1\leq i\leq k,1\leq j\leq m$, we have $\chi_i(G\otimes H)\leq km=\chi_i(G)\chi_i(H)$.
\end{proof}

For a matrix $M$ and a graph $G$, let $$M^{\otimes k}=\underbrace{M\otimes\cdots\otimes M}_k$$ and $$G^{\otimes k}=\underbrace{G\otimes\cdots\otimes G}_k$$ denote the \emph{Kronecker product of $k$ copies of $M$} and the \emph{direct product of $k$ copies of $G$}, respectively.

We denote by $\mathcal{A}_G$ the set of real matrices $A$ indexed by $V(G)$ satisfying $(A)_{uv}=0$ if $u=v$ or $u,v$ are two distinct nonadjacent vertices in $G$. 

\section{Injective chromatic number}\label{sec:boundsinjectivechrom}

Injective colorings were introduced by Hahn, Kratochvı\'il, \v{S}ir\'a\v{n} and Sotteau in \cite{01}, in which they proved the inequality $\Delta \leq \chi_i(G) \leq \Delta^2 -\Delta +1$, where $\Delta$ is the maximum degree of $G$. In the same paper the authors showed applications of the injective chromatic number of the hypercube in the theory of error-correcting codes. Injective colorings have received a lot of attention in the literature since their introduction, see e.g. \cite{83,15,12,FK2002,HRS2009, KS2014,CKY2011}.

Note that an injective colouring is not necessarily a proper colouring, i.e., it is possible for two adjacent vertices to receive the same colour. Also, observe that
$\chi(G)\leq \chi_i(G)\leq \chi(G^2)$ where $G^2$ is the \emph{square power} of $G$ (a graph with the same vertex set as $G$ and where two vertices are adjacent if they are at distance at most $2$ in $G$).

The relation between $\chi_i(G)$ and $\Delta(G)$ has been widely investigated, see e.g. \cite{40,83,12,49,01,PG2025}. Doyon, Hahn and Raspaud \cite{15} established an interesting relation between the injective chromatic number and the maximum average degree of graphs. Also the injective chromatic number of some graph operations (including direct product, Cartesian product and lexicographic product) have been studied in \cite{18,82}.

In \cite{02} the following lower bound for $\chi_i(G)$ was obtained in terms of the order, size and open packing number of $G$.  

\begin{thm}\cite[Theorem 6]{02}\label{Bresarbound}
If $G$ is a connected graph of order $n\geq2$ and size $m$, then \begin{eqnarray*}
\chi_i(G)\geq\frac{1}{2}+\sqrt{\frac{1}{4}+\frac{2m-n}{\rho^o(G)}}.
\end{eqnarray*}
\end{thm}

We are now ready to derive the first spectral bound for $\chi_i(G)$, which later we will also use to estimate $\chi_i(G^{\otimes k})$.

\begin{thm}\label{thm:evinjectivelowerbound}
Let $G$ be an $n$-vertex graph without isolated vertices, and let $A\in\mathcal{A}_G$ be a matrix associated with $G$ such that all the row vectors of $A$ are unit vectors and $AA^\top\neq I$. Then
\begin{eqnarray*}
\chi_i(G)\geq \sup_{k}\sqrt[k]{\chi_i(G^{\otimes k})} \geq \inf_{k}\sqrt[k]{\chi_i(G^{\otimes k})} \geq \frac{\lambda_1-\lambda_n}{1-\lambda_n},
\end{eqnarray*}
where $\lambda_1=\lambda_1(AA^\top)$, $\lambda_n=\lambda_n(AA^\top)$.
\end{thm}

\begin{proof}
Since $AA^\top-\lambda_nI$ is positive semidefinite, $AA^\top\neq I$ and $(AA^\top)_{uu}=1$ for all $u$, we have $1-\lambda_n>0$. Let $M=(1-\lambda_n)^{-1}(AA^\top-\lambda_nI)$. Then $M$ is positive semidefinite and $(M)_{uu}=1$ for all $u$. From the definition of $\mathcal{A}_G$, we have
\begin{eqnarray*}
(M)_{uv}=(1-\lambda_n)^{-1}(AA^\top)_{uv}=0
\end{eqnarray*}
when $u,v$ are two distinct vertices satisfying $N_G(u)\cap N_G(v)=\emptyset$.

According to Lemma \ref{lema:injectivekronecker}, we get
\begin{eqnarray*}
\chi_i(G)\geq \sup_{k}\sqrt[k]{\chi_i(G^{\otimes k})}.
\end{eqnarray*}
So we only need to prove that $\chi_i(G^{\otimes k})\geq (\frac{\lambda_1-\lambda_n}{1-\lambda_n})^k$ for any positive integer $k$. Let $H_k$ be the two-step graph of the direct product graph $G^{\otimes k}$. Using the sandwich inequality, we have
\begin{eqnarray*}
\chi_i(G^{\otimes k})=\chi(H_k)\geq \vartheta(\overline{H_k}).
\end{eqnarray*}
If two distinct vertices $u=u_1\cdots u_k$ and $v=v_1\cdots v_k$ are nonadjacent in the two-step graph $H_k$, then there exists $i\in\{1,\ldots,k\}$ such that $N_G(u_i)\cap N_G(v_i)=\emptyset$ in $G$. In this case, we have $(M)_{u_iv_i}=0$ and
\begin{eqnarray*}
(M^{\otimes k})_{uv}=(M)_{u_1v_1}\cdots (M)_{u_kv_k}=0. 
\end{eqnarray*}
Since $M$ is a positive semidefinite matrix and $(M)_{uu}=1$ for all $u\in V(G)$, $M^{\otimes k}$ is also positive semidefinite matrix and each diagonal entry of $M^{\otimes k}$ is $1$. By Lemma \ref{lema:balla}, we have
\begin{eqnarray*}
\chi_i(G^{\otimes k})\geq\vartheta(\overline{H_k})\geq\lambda_1(M^{\otimes k}) = \lambda_1(M)^k =\left(\frac{\lambda_1-\lambda_n}{1-\lambda_n}\right)^k,
\end{eqnarray*}
which implies
\begin{eqnarray*}
\inf_{k}\sqrt[k]{\chi_i(G^{\otimes k})} \geq \frac{\lambda_1-\lambda_n}{1-\lambda_n}.
\end{eqnarray*}
Thus
\begin{align*}
\chi_i(G)&\geq \sup_{k}\sqrt[k]{\chi_i(G^{\otimes k})} \geq \inf_{k}\sqrt[k]{\chi_i(G^{\otimes k})} \geq \frac{\lambda_1-\lambda_n}{1-\lambda_n}.
 \end{align*}
\end{proof}

As a direct application of Theorem \ref{thm:evinjectivelowerbound} we obtain the following corollary.
\begin{cor}\label{cor3.3}
For an $n$-vertex graph $G$ without isolated vertices, let $A=D_G^{-\frac{1}{2}}A_G$, where $D_G$ is the diagonal degree matrix of $G$. Then
\begin{eqnarray*}
\chi_i(G)\geq \sup_{k}\sqrt[k]{\chi_i(G^{\otimes k})} \geq \inf_{k}\sqrt[k]{\chi_i(G^{\otimes k})} \geq \frac{\lambda_1-\lambda_n}{1-\lambda_n},
\end{eqnarray*}
where $\lambda_1=\lambda_1(D_G^{-\frac{1}{2}}A_G^2D_G^{-\frac{1}{2}})$, $\lambda_n=\lambda_n(D_G^{-\frac{1}{2}}A_G^2D_G^{-\frac{1}{2}})$.
\end{cor}

Next we show a graph class for which our lower bound from Corollary \ref{cor3.3} outperforms the known bound from Theorem \ref{Bresarbound}(\cite[Theorem 6]{02}). We also provide a graph family that achieves the bound of Corollary \ref{cor3.3}, thus demonstrating its sharpness.

\begin{example}
Let $G=K_{1,n}$ be a star graph of order $n+1$ and size $n$, then we know $\chi_i(G)=n$ and $\rho^o(G)=2$. By Theorem \ref{Bresarbound} we get
$$\chi_i(G)\geq \frac{1}{2}+\sqrt{\frac{1}{4}+\frac{2|E(G)|-|V(G)|}{\rho^o(G)}}=\frac{\sqrt{2n-1}+1}{2}.$$
Normalizing the row vectors of the adjacency matrix $A_G$ to obtain the matrix $A$, then we have $\lambda_1(AA^\top)=n$ and $\lambda_{n+1}(AA^\top)=0$. Our lower bound in Corollary \ref{cor3.3} is
$$\frac{\lambda_1(AA^\top)-\lambda_{n+1}(AA^\top)}{1-\lambda_{n+1}(AA^\top)}=n > \frac{\sqrt{2n-1}+1}{2}.$$
For the complete bipartite graph $H=K_{m,n}$, the direct product $H^{\otimes k}$ is also a graph class attaining the bounds in Theorem \ref{thm:evinjectivelowerbound}.
Take $A=D^{-\frac{1}{2}}A_H$, where $D$ is the diagonal degree matrix of $H$. Then we have $\lambda_1(AA^\top)=\max\{m,n\}$ and $\lambda_{m+n}(AA^\top)=0$. Then
$$\chi_i(G)=\max\{m,n\}= \frac{\lambda_1(AA^\top)-\lambda_{m+n}(AA^\top)}{1-\lambda_{m+n}(AA^\top)}.$$
By Theorem \ref{thm:evinjectivelowerbound}, we have
$$\chi_i(G^{\otimes k})=\left(\frac{\lambda_1(AA^\top)-\lambda_{m+n}(AA^\top)}{1-\lambda_{m+n}(AA^\top)}\right)^k=(\max\{m,n\})^k$$
for any positive integer $k$.
\end{example}

We will further compare the lower bound of Corollary \ref{cor3.3} with Theorem \ref{Bresarbound} (\cite[Theorem 6]{02}) in Table \ref{graph6 5vertices}. In particular, we provide the proportion of graphs for which our bound is strictly better among all connected non isomorphic graphs of small order. 

\begin{rem}
Suppose that $A \in \mathcal{A}_G$ is a matrix satisfying the conditions in Theorem \ref{thm:evinjectivelowerbound}. Let $\lambda_1$ and $\lambda_n$ be the largest and smallest eigenvalues of the matrix $AA^\top$, respectively. If $\chi_i(G)= \frac{\lambda_1-\lambda_n}{1-\lambda_n}$, then by Theorem \ref{thm:evinjectivelowerbound}, we can obtain
\begin{eqnarray*}
\chi_i(G^{\otimes k})=\chi_i(G)^k= \left(\frac{\lambda_1-\lambda_n}{1-\lambda_n}\right)^k
\end{eqnarray*}
for any positive integer $k$.
\end{rem}

For an $n$-vertex graph $G$, let $\lambda(G)=\min_{1\leq i\leq n}|\lambda_i(A_G)|$.
\begin{cor}\label{coro:injectiveboundregular}
Let $G$ be a regular graph of degree $d>0$. Then
\begin{eqnarray*}
\chi_i(G)\geq \sup_{k}\sqrt[k]{\chi_i(G^{\otimes k})} \geq \inf_{k}\sqrt[k]{\chi_i(G^{\otimes k})}\geq\frac{d^2-\lambda(G)^2}{d-\lambda(G)^2}.
\end{eqnarray*}
\end{cor}
\begin{proof}
Let $A=d^{-\frac{1}{2}}A_G$. Then $\lambda_1(AA^\top)=d$, $\lambda_n(AA^\top)=d^{-1}\lambda(G)^2$. The result now follows from Theorem \ref{thm:evinjectivelowerbound}.
\end{proof}

Table \ref{table:knowngraphs} in the Appendix illustrates that, for some Sagemath named graphs, our lower bound in Corollary \ref{coro:injectiveboundregular} can outperform the maximum degree bound, as well as the bound in Theorem \ref{Bresarbound} (\cite[Theorem 6]{02}). 

\section{Open packing number}\label{sec:boundopenpackingnumber}

Injective colorings of graphs are related to the so-called open packings, cf. \cite{PG2021} More precisely, an injective coloring of a graph is exactly a partition of its vertex set into open packings. A vertex subset $S\subseteq V(G)$ is called an \textit{open packing} of $G$ if any two vertices in $S$ do not have a common neighbor in $G$, that is, $N_G(u)\cap N_G(v)= \emptyset$ for any distinct vertices $u,v\in S$. As mentioned, there is a one-to-one correspondence between injective $k$-colorings of $G$ and $k$-partitions of $V(G)$ into $k$ open packings. So the injective chromatic number $\chi_i(G)$ is the minimum $k$ such that $G$ has a partition $V(G)=S_1\cup\cdots\cup S_k$, where $S_1,\ldots,S_k$ are open packings of $G$. The \textit{open packing number} of $G$, introduced in \cite{20}  and denoted $\rho^o(G)$, is the maximum cardinality among all open packings of $G$. Previous work on the open packing number can be found, for instance, in \cite{10,09,21,BKR2024}.

We start this section by showing an upper bound for the open packing number in terms of the eigenvalues of an associated matrix to a graph.

\begin{thm}\label{thm:evopenpackingnumberbound}
Let $A \in \mathcal{A}_G$ be a matrix associated with a graph $G$, and let $\lambda$ be a positive eigenvalue of $A$ with a positive unit eigenvector $x$. Then, the open packing number of $G$ satisfies
\begin{eqnarray*}
\rho^o(G)\leq\lambda^{-2}\max_{u\in V(G)}\frac{(AA^\top)_{uu}}{x_u^2},
\end{eqnarray*}
with equality if and only if there is a constant $c$ and there exists an open packing $S$ such that
\begin{equation}
cx_u=\sum_{v\in N_G(u)\cap S}(A)_{vu}x_v^{-1}~~~~~~(u\in V(G))\tag{4.1}
\end{equation}
and
\begin{equation}
\frac{(AA^\top)_{vv}}{x^2_v} = \max_{u\in V(G)}\frac{(AA^\top)_{uu}}{x^2_u}\tag{4.2}
\end{equation}
for all $v\in S$.
\end{thm}

\begin{proof}
By $Ax=\lambda x$ and $\lambda>0$, we obtain $x=\lambda^{-1}Ax$. For any open packing $S$ of $G$, let $y=(y_1,\ldots,y_n)^\top$ be the vector such that
\begin{eqnarray*}
y_u=\begin{cases}x_u^{-1}~~~~~~~~~~~~~\mbox{if}~u\in S,\\
0~~~~~~~~~~~~~~~~\mbox{if}~u\notin S.\end{cases}
\end{eqnarray*}
By the Cauchy-Schwarz inequality, we have
\begin{eqnarray*}
|S|^2=(y^\top x)^2=(\lambda^{-1}y^\top Ax)^2\leq\lambda^{-2}y^\top AA^\top y.
\end{eqnarray*}
For any two distinct vertices $u,v$ in $S$, we have $N_G(u)\cap N_G(v)= \emptyset$. So $(AA^\top)_{uv}= 0$ if $u,v\in S$ and $u\neq v$. Then we get
\begin{eqnarray*}
|S|^2\leq\lambda^{-2}y^\top AA^\top y=\lambda^{-2}\sum_{u\in S}\frac{(AA^\top)_{uu}}{x_u^2}.
\end{eqnarray*}
To show the equality characterization, suppose that $S$ is a maximum open packing of $G$. Then $\rho^o(G)=|S|$ and
\begin{eqnarray*}
\rho^o(G)\leq\lambda^{-2}\max_{u\in V(G)}\frac{(AA^\top)_{uu}}{x_u^2},
\end{eqnarray*}
with equality if and only if there is a constant $c$ such that $A^\top y=cx$ and $$\frac{(AA^\top)_{vv}}{x^2_v} = \max_{u\in V(G)}\frac{(AA^\top)_{uu}}{x^2_u}$$ for all $v\in S$.
\end{proof}

If $A \in \mathcal{A}_G$ is the adjacency matrix of a graph $G$, then we obtain the following corollary.
\begin{cor}\label{cor:openpacking}
Let $\lambda$ be the largest eigenvalue of the adjacency matrix $A_G$ with a positive unit eigenvector $x$, and let $d_u$ be the degree of vertex $u$. Then
$$\rho^o(G) \leq \lambda^{-2} \max_{u\in V(G)}\frac{d_u}{x_u^2},$$
with equality if and only if there exists an open packing $S$ that satisfies conditions (4.1) and $\frac{d_v}{x^2_v} = \max_{u\in V(G)}\frac{d_u}{x^2_u}$ for all $v\in S$.
\end{cor}

\begin{proof}
We know $(AA^\top)_{uu}=d_u$. By Theorem \ref{thm:evopenpackingnumberbound}, we get
$$\rho^o(G) \leq \lambda^{-2} \max_{u\in V(G)}\frac{d_u}{x_u^2},$$
with equality if and only if there exists an open packing $S$ that satisfies conditions (4.1) and $\frac{d_v}{x^2_v} = \max_{u\in V(G)}\frac{d_u}{x^2_u}$ for all $v\in S$.
\end{proof}

Using Theorem \ref{thm:evopenpackingnumberbound}, next we show spectral upper bounds for the open packing number of the direct product graph $G^{\otimes k}$.

\begin{thm}\label{thm:openpackingkronecker}
Let $A \in \mathcal{A}_G$ be a matrix associated with a graph $G$, and let $\lambda$ be a positive eigenvalue of $A$ with a positive unit eigenvector $x$. Then
$$\rho^o(G)\leq \sup_{k}\sqrt[k]{\rho^o(G^{\otimes k})}\leq \lambda^{-2} \max_{u\in V(G)}\frac{(AA^\top)_{uu}}{x_u^2}.$$
\end{thm}

\begin{proof}
Since $A^{\otimes k}\in \mathcal{A}_{G^{\otimes k}}$ is the matrix associated with the direct product graph $G^{\otimes k}$. Based on the characteristic equation $A x=\lambda x$, we have
$$A^{\otimes k}x^{\otimes k}=\lambda^kx^{\otimes k}.$$
So $x^{\otimes k}$ is the positive unit eigenvector of matrix $A^{\otimes k}$ associated with the positive eigenvalue $\lambda^k$.
Now using Theorem \ref{thm:evopenpackingnumberbound}, we have
$$\rho^o(G^{\otimes k}) \leq \lambda^{-2k} \bigg(\max_{u\in V(G)}\frac{(AA^\top)_{uu}}{x_u^2}\bigg)^k.$$
Next, suppose that $S$ is the maximum open packing of $G$ such that $|S|=\rho^o(G)$. Then $S^k=\{u_1\cdots u_k:u_1,\ldots,u_k\in S\}$ is an open packing of direct product graph $G^{\otimes k}$. So we have ${\rho^o(G)}^k \leq \rho^o(G^{\otimes k})$.
Hence we obtain
\begin{align*}
&\rho^o(G)\leq \sup_{k}\sqrt[k]{\rho^o(G^{\otimes k})}\leq \lambda^{-2} \max_{u\in V(G)}\frac{(AA^\top)_{uu}}{x_u^2}.
\end{align*}
\end{proof}

Below we present a family of direct product graphs attaining the bound in Theorem \ref{thm:openpackingkronecker}.

\begin{example}
For a complete bipartite graph $G= K_{m,n} (m\leq n)$, we know that $\rho^o(G)=2$. Let $A=A_G$ be the adjacency matrix of graph $G$, and let $\lambda=\sqrt{mn}$ be the largest eigenvalue of $A$ with a positive unit eigenvector $x$. Then $x_u=\frac{1}{\sqrt{2m}}$ if $u$ is a vertex belongs to the color class of size $m$, and $x_u=\frac{1}{\sqrt{2n}}$ if $u$ is a vertex belongs to the color class of size $n$. Then $$\max_{u\in V(G)}\frac{(AA^\top)_{uu}}{x_u^2}=2mn$$ and
$$\rho^o(G)=\lambda^{-2}\max_{u\in V(G)}\frac{(AA^\top)_{uu}}{x_u^2}=2.$$
By Theorem \ref{thm:openpackingkronecker}, we have
$$\rho^o(G^{\otimes k})=\rho^o(G)^k=\bigg( \lambda^{-2} \max_{u\in V(G)}\frac{d_u}{x_u^2}\bigg)^k=2^k$$
for any positive integer $k$.
\end{example}

\section{Injective chromatic index}\label{section:Injectivechromaticindex}

Injective edge colorings were introduced in \cite{Cardosofilomat} motivated by the socalled Packet Radio Network problem. The injective chromatic index is also called induced star arboricity in \cite{Axenovich}. Bounds for $\chi_i'(G)$ have been obtained in the literature, see e.g. \cite{Bradshaw,Cardosofilomat,KRX2021}.

For an induced subgraph $H$ of a graph $G$, we say that $H$ is an \textit{induced star forest} if each component of $H$ is a star. To derive our bounds in this section, we will use the following known result.
\begin{lem}\cite[Proposition 2.2]{Cardosofilomat}\label{induced star}
Let $G$ be a graph. Then $\chi_i'(G)=k$ if and only if $k$ is the minimum positive integer for which $E(G)$ can be partitioned into non-empty subsets $E_1,\ldots,E_k$ such that $E_i$ induces an induced star forest for every $i$.
\end{lem}

In \cite[Proposition 8]{KRX2021}, Kostochka et al. proved that $\chi_i'(G)\geq d$ when $G$ is a $d$-regular bipartite graph. For a $d$-regular bipartite graph $G$, the minimum eigenvalue of $A_G$ is $-d$. We extend \cite[Proposition 8]{KRX2021} as follows.

\begin{thm}\label{ICIbound}
Let $G$ be an $n$-vertex graph with $m\geq1$ edges and minimum degree $\delta$, and let $\lambda_1\geq\cdots\geq\lambda_n$ be eigenvalues of $A_G$. Then
\begin{eqnarray*}
\chi_i'(G)\geq\frac{m(\delta^2-\lambda_1\lambda_n)}{-n\lambda_1\lambda_n}.
\end{eqnarray*}
Moreover, if $G$ is $d$-regular, then
\begin{eqnarray*}
\chi_i'(G)\geq\frac{d(d-\lambda_n)}{-2\lambda_n}.
\end{eqnarray*}
\end{thm}

\begin{proof}
By Lemma \ref{induced star}, we know that there exists a partition $E(G)=E_1\cup\cdots\cup E_k$ such that $k=\chi_i'(G)$ and $E_i$ induces an induced star forest for every $i$. Hence $\max_{1\leq i\leq k}|E_i|\leq\alpha(G)$, where $\alpha(G)$ is independence number of $G$. Then
\begin{eqnarray*}
\chi_i'(G)\alpha(G)\geq m.
\end{eqnarray*}
By Lemma \ref{Haemersbound}, we have
\begin{eqnarray*}
\chi_i'(G)\geq\frac{m(\delta^2-\lambda_1\lambda_n)}{-n\lambda_1\lambda_n}.
\end{eqnarray*}
If $G$ is $d$-regular, then $\lambda_1=d$, $m=nd/2$ and
\begin{eqnarray*}
\frac{m(\delta^2-\lambda_1\lambda_n)}{-n\lambda_1\lambda_n}=\frac{d(d-\lambda_n)}{-2\lambda_n}.
\end{eqnarray*}
\end{proof}

Next we show a class of graphs attaining the bound in Theorem \ref{ICIbound}.
\begin{example}
Let $G=K_{n,\ldots,n}$ be the complete $p$-partite graph with $pn$ vertices. Then $G$ is regular of degree $(p-1)n$ and the minimum eigenvalue of $A_G$ is $-n$. By Theorem \ref{ICIbound}, we have
\begin{eqnarray*}
\chi_i'(G)\geq\frac{p(p-1)n}{2}.
\end{eqnarray*}
Since $E(G)$ can be partitioned into $\frac{p(p-1)n}{2}$ induced star forests, we have 
\begin{eqnarray*}
\chi_i'(G)=\frac{p(p-1)n}{2}.
\end{eqnarray*}
\end{example}

In \cite[Proposition 3.1]{Cardosofilomat} the following lower bound on the injective chromatic index of $G$ in terms of clique number $\omega(G)$ was obtained. Later, we will compare it with Theorem \ref{ICIbound}.
\begin{thm}\cite[Proposition 3.1]{Cardosofilomat}\label{Cardosobound}
For any connected graph $G$ of order $n \geq 2$, it holds that
\begin{eqnarray*}
\chi_i'(G)\geq \frac{\omega(G)(\omega(G)-1)}{2}.
\end{eqnarray*}
\end{thm}

In Table \ref{table:injectivechromaticindex} in the Appendix we compare Theorem \ref{ICIbound} and Theorem \ref{Cardosobound} (\cite[Proposition 3.1]{Cardosofilomat}) for some Sagemath named graphs and for some graph classes. The computational experiments show that our spectral bound on the injective chromatic index is also tight, and that for some graphs, our bound is better than \cite[Proposition 3.1]{Cardosofilomat}.

\section{Strong chromatic index}\label{section:Strongchromaticindex}

Given a graph $G$, the square of the line graph is denoted $L(G)^2$. Alternatively, the strong chromatic index of $G$ is the chromatic number of $L(G)^2$, and according to \eqref{eq:relationstrongandchromaticnumbers}, we have $\chi_s'(G)=\chi{(L(G)^2)}\geq \chi_i(L(G))$. Also, note that from the definitions, it holds $\chi_i'(G)\leq\chi_s'(G)$ for any graph $G$, where $\chi_s'(G)$ is strong chromatic index of $G$.

 A trivial upper bound of the strong chromatic index is $\chi_s'(G)\leq 2\Delta^2-2\Delta+1$. More refined upper bounds on the strong chromatic index have also been obtained in the literature, see e.g. \cite{FSGT1989,MR1997,BJ2018}. Erd\H{o}s and Ne\v{s}et\v{r}il \cite{FSGT1989} posed the following conjecture on the parameter of interest. 

\begin{conjecture}\cite[Section1]{FSGT1989}\label{conjecture}
\[
\begin{cases}
    \chi_s'(G) \leq\frac{5}{4} \Delta^2,~~~~~~~~~~~~~~~~~~~~~~if ~\Delta~ is~ even,\\
    \chi_s'(G)\leq \frac{5\Delta^2-2\Delta+1}{4},~~~~~~~~~~~~~~~~if ~\Delta~ is~ odd.
\end{cases}
\]
\end{conjecture}

The conjecture above has triggered quite some work, see e.g. \cite{BJ2018,Andersen1992,FKS2005,Mahdian2000}. 

Most existing results on the strong chromatic index consist of upper bounds. In \cite[Corollary 9]{AR2026}, a spectral lower bound for the strong chromatic index of a $k$-regular graph was shown using the adjacency eigenvalues of the line graph $L(G)$.

\begin{cor}\cite[Corollary 9]{AR2026}\label{AR2026}
Let $G$ be a $k$-regular graph. Let $2(k-1)=\theta'_0 >\cdots >\theta'_{d'}$ be
the distinct adjacency eigenvalues of $L(G)$, with $d' \geq 2$. Let $\theta'_i$ be the largest eigenvalue such that $\theta'_i \leq -1$. Then
\begin{eqnarray*}
\chi_s'(G) \geq \frac{\mid E\mid}{\big\lfloor\mid E \mid \frac{\theta'_0+\theta'_i\theta'_{i-1}}{(\theta'_0-\theta'_i)(\theta'_0-\theta'_{i-1})}\big\rfloor}.
\end{eqnarray*}
\end{cor}

In this section we show two alternative spectral lower bounds on the strong chromatic index. Furthermore, we give some graph classes for which these bounds are tight.

\begin{thm}\label{SCIbound1}
Let $G$ be an $n$-vertex graph with $m$ edges and adjacency eigenvalues $\lambda_1\geq\cdots\geq\lambda_n$. Then
\begin{eqnarray*}
\chi_s'(G)\geq\frac{m}{\min\{|\{i:\lambda_i\geq1\}|,|\{i:\lambda_i\leq-1\}|\}}.
\end{eqnarray*}
\end{thm}
\begin{proof}
There exists a partition $E(G)=E_1\cup\cdots\cup E_k$ such that $k=\chi_s'(G)$ and each $E_i$ is an induced matching of $G$. By the Cauchy interlacing theorem, we have
\begin{eqnarray*}
|E_i|\leq \min\{|\{i:\lambda_i\geq1\}|,|\{i:\lambda_i\leq-1\}|\}
\end{eqnarray*}
for $i=1,\ldots,k$. Hence
\begin{eqnarray*}
m=\sum_{i=1}^k|E_i|\leq k\min\{|\{i:\lambda_i\geq1\}|,|\{i:\lambda_i\leq-1\}|\}.
\end{eqnarray*}
\end{proof}

\begin{thm}\label{SCIbound2}
Let $G$ be an $n$-vertex $d$-regular graph with adjacency eigenvalues $\lambda_1\geq\cdots\geq\lambda_n$. Then
\begin{eqnarray*}
\chi_s'(G)\geq\frac{nd}{2\left\lfloor\frac{n(1-\lambda_n)}{2(d-\lambda_n)}\right\rfloor}\geq\frac{d(d-\lambda_n)}{1-\lambda_n}.
\end{eqnarray*}
\end{thm}

\begin{proof}
There exists a partition $E(G)=E_1\cup\cdots\cup E_k$ such that $k=\chi_s'(G)$ and each $E_i$ is an induced matching of $G$. Let $r=\max_{1\leq i\leq k}|E_i|$. Then $k\geq\frac{nd}{2r}$ and $G$ has a $1$-regular induced subgraph with $2r$ vertices. By Lemma \ref{subgraphbound}, we have
\begin{eqnarray*}
n-2rd&\geq&(n-2r)\lambda_n,\\
r&\leq&\left\lfloor\frac{n(1-\lambda_n)}{2(d-\lambda_n)}\right\rfloor.
\end{eqnarray*}
Hence
\begin{eqnarray*}
k\geq\frac{nd}{2r}\geq\frac{nd}{2\left\lfloor\frac{n(1-\lambda_n)}{2(d-\lambda_n)}\right\rfloor}\geq\frac{d(d-\lambda_n)}{1-\lambda_n}.
\end{eqnarray*}
\end{proof}

It is easy to see that for complete graphs and complete bipartite graphs, the bound in Theorem \ref{SCIbound1} is tight. For complete graphs, the bound in Theorem \ref{SCIbound2} is also achieved.

In Table \ref{table:strongchromaticindex-dregular} we compare Conjecture \ref{conjecture} (\cite[Section 1]{FSGT1989}), Theorem \ref{SCIbound1}, Theorem \ref{SCIbound2} and Corollary \ref{AR2026} (\cite[Corollary 9]{AR2026}) for the Sagemath named graphs. The computational results show that our spectral lower bounds on the strong chromatic index can be tight and, in some specific instances, can bring improvements upon the spectral lower bound from \cite[Corollary 9]{AR2026}. 

\section{Applications}\label{sec:apps}
Hypercubes appear throughout theoretical computer science and combinatorics for a variety of reasons, making their structure an essential subject of study. Despite their seemingly simple form, they often give rise to surprisingly complex and challenging problems; we will see two of them in this section. 

\subsection{The code covering number of hypercubes}\label{sec:codecoverignumber}

In coding theory, there is an interest in the problem of determining the injective
chromatic number of the $n$-dimensional hypercube, see e.g. \cite{BB1977,DHLL1990,Mnupublished}. Recall that the \emph{$n$-dimensional cube} $Q_n$ is the graph defined on the vertex set $\{0,1\}$ by $[a,b] \in E(Q_n)$ if and only if $a$ and $b$ differ in exactly one coordinate. A (binary) \emph{code} of length $n$ is an arbitrary subset $S$ of vertices of the $n$-cube $Q_n$. The code $S$ is \emph{single-error-correcting} if the Hamming distance of any two distinct vertices of $S$ is at least $3$. The \emph{code covering number} of the $n$-cube, denoted $\gamma(Q_n)$, is the minimum number $t$ of single-error-correcting codes $S_1, \dots, S_t$ such that $V(Q_n) = S_1\cup \cdots \cup S_t$.

Let us first give a small overview of the some known results on the code covering number of the $n$-cube, obtained by Hahn, Kratochvı\'il,  \v{S}ir\'a\v{n} and Sotteau \cite{01}.

\begin{lem}\cite{01}\label{lem5.1} $\gamma(Q_n)=\chi_i(Q_{n+1})$.
\end{lem}

\begin{lem}\cite{01} \label{lem5.2}
$\chi_i(Q_n)=n$ if and only if $n=2^r$ for some integer $r$.
\end{lem}

\begin{lem}\cite{01} \label{lem5.3}
$\chi_i(Q_{2^r-j})=2^r$ for $0\leq j\leq3$.
\end{lem}

\begin{lem}\cite{01} \label{lem5.4}
$\chi_i(Q_{2n+1})\leq2\chi_i(Q_{n+1})$.
\end{lem}

Thus, for any positive integer $n$, it is known that $\chi_i(Q_n)\geq n$, with equality if and only if $n=2^r$ for some integer $r$ (see Lemma \ref{lem5.2}). For an odd integer $n>1$, we prove a sufficient and necessary conditions for $\chi_i(Q_n)=n+1$. 

\begin{thm}\label{thm5.6}
Let $n>1$ be an odd integer. Then $\chi_i(Q_n)\geq n+1$, with equality if and only if $n=2^r-1$ for some integer $r$.
\end{thm}
\begin{proof}
Recall that $\chi_i(Q_n)$ equals to the chromatic number $\chi(Q_n^{(2)})$ of the two-step graph $Q_n^{(2)}$. Notice that there are no triangles in $Q_n$, and any two vertices with distance $2$ in $Q_n$ have exactly $2$ common neighbors. So the adjacency matrices of $Q_n$ and $Q_n^{(2)}$ satisfy
\begin{eqnarray*}
A_{Q_n^{(2)}}=\frac{1}{2}(A_{Q_n}^2-nI).
\end{eqnarray*}
Since $n$ is odd, the minimum eigenvalue of $A_{Q_n}^2$ is $1$. Then $Q_n^{(2)}$ is an regular graph with degree $\frac{n^2-n}{2}$, and the minimum adjacency eigenvalue of $Q_n^{(2)}$ is $\frac{1-n}{2}$. By Lemma \ref{HoffmanChi} we have
\begin{eqnarray*}
\chi_i(Q_n)=\chi(Q_n^{(2)})\geq1+\frac{n^2-n}{n-1}=n+1.
\end{eqnarray*}
If $\chi_i(Q_n)=\chi(Q_n^{(2)})=n+1$, then by Lemma \ref{BlokhuisProposition 2.3}, we know that $Q_n^{(2)}$ has a proper vertex coloring such that all color classes have equal size. In this case, $2^n$ is divisible by $n+1$, i.e., $n=2^r-1$ for some integer $r$. Conversely, if $n=2^r-1$ for some integer $r$, then by Lemma \ref{lem5.3}, we have $\chi_i(Q_n)=n+1$.
\end{proof}

\begin{thm}\label{thm5.7}
Suppose that $n=2^r-j$ for $r\geq3$ and $3\leq j\leq4$. Then
\begin{eqnarray*}
2n+3\leq\chi_i(Q_{2n+1})\leq 2n+2j.
\end{eqnarray*}
\end{thm}
\begin{proof}
By Lemma \ref{lem5.4} and Lemma \ref{lem5.3}, we have
\begin{eqnarray*}
\chi_i(Q_{2n+1})\leq2\chi_i(Q_{n+1})=2^{r+1}=2n+2j.
\end{eqnarray*}
Since $2n+1=2^{r+1}-(2j-1)$, using Theorem \ref{thm5.6} we obtain
\begin{eqnarray*}
\chi_i(Q_{2n+1})\geq2n+3.
\end{eqnarray*}
\end{proof}

Now we are ready to show the following result concerning the code covering number of the $n$-cube $\gamma(Q_n)$, which provides sufficient and necessary conditions for $\gamma(Q_n)=n+2$, and new lower bounds for the case of $n=2^r-j$ ($r\geq 3$, $3\leq j\leq 4$).

\begin{thm}\label{thm:codecoveringspectralbound}
Let $Q_n$ denote the $n$-dimensional cube.
\begin{description}
    \item[$(i)$] If $n$ is odd, then
$\gamma(Q_n)\geq n+1$, with equality if and only if $n=2^r-1$ for some integer $r$.
 \item[$(ii)$] If $n$ is even, then
$\gamma(Q_n)\geq n+2$, with equality if and only if $n=2^r-2$ for some integer $r$.
\item[$(iii)$] If $n=2^r-j$ for $r\geq3$ and $3\leq j\leq4$, then
\begin{eqnarray*}
2n+3\leq\gamma(Q_{2n})\leq 2n+2j.
\end{eqnarray*}
\end{description}
\end{thm}

\begin{proof}
For a graph $G$, let $\lambda(G)$ denote the minimum absolute value among all eigenvalues of $A_{G}$.
\leavevmode
\begin{description}
    \item[$(i)$] If $n$ is odd, then $\lambda(Q_{n+1})=0$. By Lemma \ref{lem5.1} and Corollary \ref{coro:injectiveboundregular}, we obtain
    \begin{eqnarray*}
    \gamma(Q_n)=\chi_i(Q_{n+1})\geq \frac{(n+1)^2-0}{(n+1)-0}=n+1.
    \end{eqnarray*}
    Lemma \ref{lem5.2} implies that the equality holds if and only if $n+1=2^r$ for some integer $r$.
    \item[$(ii)$] If $n$ is even, then $\lambda(Q_{n+1})=1$. By Lemma \ref{lem5.1} and Corollary \ref{coro:injectiveboundregular}, we get  
     \begin{eqnarray*}
    \gamma(Q_n)=\chi_i(Q_{n+1})\geq \frac{(n+1)^2-1}{(n+1)-1}=n+2. 
    \end{eqnarray*}
    Theorem \ref{thm5.6} implies that the equality holds if and only if $n+1=2^r-1$ for some integer $r$.
    \item[$(iii)$] According to Lemma \ref{lem5.1} and Theorem \ref{thm5.7}, we have
    \begin{eqnarray*}
    2n+3\leq \gamma(Q_{2n})=\chi_i(Q_{2n+1})\leq 2n+2j.
    \end{eqnarray*}
\end{description}
\end{proof}

\subsection{The open packing number of bipartite prisms}\label{sec:bipartiteprisms}

Recall that, given two graphs $G$ and $H$, their \emph{Cartesian product} $G\Box H$ is the graph with vertex set $V(G)\times V(H)$, and two vertices $(u_1, v_1)$ and $(u_2,v_2)$ are adjacent in $G\Box H$ if and only if $\{u_1,u_2\}\in E(G),v_1=v_2$ or $u_1=u_2,\{v_1,v_2\}\in E(H)$. The \emph{square power} of a graph $G$, is the graph with the same vertex set as $G$, and where two edges are adjacent if they are at distance at most $2$ in $G$. Let $\alpha_2(G)$ denote the independence number of the square power of $G$.

In this section we consider a generalization of hypercube graphs. The \emph{bipartite prism} of a bipartite graph $G$ is defined as $G \Box K_2$, that is, the Cartesian product of the bipartite graph $G$ with the complete graph on 2 vertices. Since $Q_n=Q_{n-1}\Box K_2$, the hypercube $Q_n$ is a special bipartite prism.

The goal is to investigate the open packing number of  bipartite prisms, with a particular focus on hypercubes. In this direction, Bre\v{s}ar et al. \cite{BKR2024} showed the following results.

\begin{thm}\cite[Theorem 2.2]{BKR2024}\label{thm:rho0alpha2}
If $G$ is a bipartite graph, then $\rho^o(G\Box K_2)=2\alpha_2(G)$.
\end{thm}

\begin{thm}\cite[Theorem 3.2]{BKR2024}
If $n \geq 2$, then $\rho^o(Q_n)\geq2^{n-\lfloor\log(n-1)\rfloor-1}$.
\end{thm}
A graph is said to be \emph{perfect injectively colorable} if it has an injective coloring in which every color class forms an open packing of largest cardinality. It is known that the $n$-dimensional hypercubes $Q_n$, for $n\leq 8$, are perfect injectively colorable graphs \cite[Section 4]{BKR2024}. So $Q_9$ is the first instance of a hypercube, which is not in this class of graphs, for which bounds have been derived in the literature. In particular,  Bre\v{s}ar et al. (see \cite[Table 1]{BKR2024}) showed that
\begin{eqnarray*}
34\leq\rho^o(Q_9)=2\alpha_2(Q_8)\leq60.
\end{eqnarray*}

The following lemma can be directly derived from \cite[Theorem 3.2]{Abiad-Zhou}, and will be used to show our last main result.

\begin{lem}\label{alpha2}\cite{Abiad-Zhou}
For a graph $G$, let $\rho>0$ be the maximum eigenvalue of $A_G$ with a positive unit eigenvector $y$. Let $p(x)=x^2+c_1x+c_0$ be a polynomial such that $p(\rho)>0$ and $p(A_G)$ is positive semidefinite. Then
\begin{eqnarray*}
\alpha_2(G)\leq\frac{1}{p(\rho)}\max_{u\in V(G)}\frac{(p(A_G))_{uu}}{y_u^2}.
\end{eqnarray*}
\end{lem}

Let $d_u$ denote the degree of a vertex $u$. We first show the following spectral upper bound for the open packing number of general bipartite prisms.
\begin{pro}\label{propo:openpackingbipartiteprisms}
Let $G$ be a bipartite graph, and let $\rho>0$ be the maximum eigenvalue of $A_G$ with a positive unit eigenvector $y$. Then
\begin{eqnarray*}
\rho^o(G\Box K_2)=2\alpha_2(G)\leq\frac{2}{\rho(\rho+\lambda)}\max_{u\in V(G)}\frac{d_u}{y_u^2},
\end{eqnarray*}
where $\lambda$ is the minimum positive eigenvalue of $A_G$. Moreover, if $G$ is $d$-regular, then
\begin{eqnarray*}
\rho^o(G\Box K_2)=2\alpha_2(G)\leq\frac{2|V(G)|}{(d+\lambda)}.
\end{eqnarray*}
\end{pro}
\begin{proof}
Since $G$ be a bipartite, $-\lambda$ is the maximum negative eigenvalue of $A_G$. Let $p(x)=x^2+\lambda x$. Then $p(\rho)>0$ and $p(A_G)$ is positive semidefinite. By Lemma \ref{alpha2}, we have
\begin{eqnarray*}
\alpha_2(G)\leq\frac{1}{\rho(\rho+\lambda)}\max_{u\in V(G)}\frac{d_u}{y_u^2}.
\end{eqnarray*}
\end{proof}
From Proposition \ref{propo:openpackingbipartiteprisms}, we obtain the following corollary.
\begin{cor}\label{Cor6.4}
$\alpha_2(Q_n)\leq \Big(n+\frac{3+(-1)^n}{2}\Big)^{-1}2^n$.
\end{cor}

In \cite[Table 1]{BKR2024}, Bre\v{s}ar et al. proved that
\begin{eqnarray*}
34\leq\rho^o(Q_9)=2\alpha_2(Q_8)\leq60.
\end{eqnarray*}

Using Corollary \ref{Cor6.4} we obtain 
$$\alpha_2(Q_8)\leq\lfloor(8+2)^{-1}2^8\rfloor=25.$$ This improves the known upper bounds by Bre\v{s}ar et al. \cite{BKR2024} to $\alpha_2(Q_8)\leq25$ and $\rho^o(Q_9)\leq50$.

Furthermore, \cite[Corollary 3.12]{BCG2025} gives an upper bound for the open packing number of hypercubes:
\begin{cor}\cite[Corollary 3.12]{BCG2025}\label{coro:Brimkovetal}
$\rho^o(Q_n)\leq \frac{2^n}{n}.$
\end{cor}
By Theorem \ref{thm:rho0alpha2}, we know that $\rho^o(G\Box K_2)=2\alpha_2(G)$, and thus we can conclude that Corollary \ref{Cor6.4} is better than Corollary \ref{coro:Brimkovetal} (\cite[Corollary 3.12]{BCG2025}).

\subsection*{Acknowledgements}
Aida Abiad is supported by NWO (Dutch Research Council) through the grant VI.Vidi.213.085.


\newpage
\section*{Appendix}


\begin{table}[htp!]
\centering
  \caption{For all connected non-isomorphic graphs of small fixed order $n$, we show the percentage of graphs of fix order for which our bound from Corollary \ref{cor3.3} strictly improves the known bound from Theorem \ref{Bresarbound} (\cite[Theorem 6]{02}).}\label{graph6 5vertices}
\begin{tabular}{c|r}
\hline
$n$ & Percentage of strict improvement \\
\hline
5 & $\nicefrac{7}{21}\approx 33\%$\\
6 & $\nicefrac{56}{112}\approx 50\%$\\
7 & $\nicefrac{300}{853}\approx 35\%$\\
8 & $\nicefrac{5381}{11117}\approx 48\%$\\
9 & $\nicefrac{84677}{261080}\approx 32\%$\\

\hline
\end{tabular}
\end{table}


\begin{table}[htp!]
  \caption{Comparison of Corollary \ref{coro:injectiveboundregular}, Theorem \ref{Bresarbound} (\cite[Theorem 6]{02}), and the trivial bound $\Delta(G)$ for some well-known $d$-regular graphs.}\label{table:knowngraphs}
    \centering
  {\small{
  \begin{tabular}{l|c|c|c} 
  \hline
  Graph & Corollary \ref{coro:injectiveboundregular}& Theorem \ref{Bresarbound} & degree\\
  \hline
  Bidiakis cube& 3 & 3 & 3\\
  Blanusa First Snark Graph& 4 & 4& 3\\
  Blanusa Second Snark Graph& 4& 4& 3\\
  Brinkmann graph& 5& 6& 4\\
  Cube graph $Q_3$ & 4 &  4&3\\
  Clebsch graph& 6 & 7 & 5\\
  Coxeter Graph& 4& 4& 3\\
  Cycle $C_{4k+r}$ ($0\leq r\leq2$)& $\frac{4-(2\cos\frac{2\pi k}{4k+r})^2}{2-(2\cos\frac{2\pi k}{4k+r})^2}$ & $-$ & 2\\
   Desargues graph& 4 & 4& 3\\
   Durer graph& 3 & 4 & 3\\
   Dyck graph& 4& 4 & 3\\
    F26A Graph& 4& 4& 3\\
    Flower Snark J5& 4& 4 & 3\\
  Folkman Graph& 4& 5 & 4\\
  Franklin graph& 4& 4& 3\\
  Frucht graph& 3 & 4 & 3\\
  \textbf{Heawood graph}& \textbf{7} & 5 &3\\
   Hoffman Graph& 4 & 5 & 4\\
  Holt graph& 5& 5& 4\\
   Icosahedral graph & 6 & 6 & 5\\
   Klein 7-regular Graph& 8& 8& 7\\
    $K_n$ & $n$ & $\frac{\sqrt{4n^2-8n+1}+1}{2}$ & $n-1$\\
   M$\ddot{o}$bius-Kantor graph& 4 & 4 &3\\
   Markstroem Graph& 3& 4& 3\\
   Nauru Graph& 3 & 3 & 3\\
   Octahedral graph& $4$ & 5 & 4\\
   Petersen graph & 4 & 4 & 3\\
   Robertson graph&  5& 5& 4\\
    \textbf{Shrikhande graph}& \textbf{16} & 10 &6\\
   Triangular prism & 3 & 3 & 3\\
   Tietze graph& 4 & 4 & 3\\
   Tutte-Coxeter graph& 4& 4& 3\\
   Wells graph& 6& 7 & 5\\ 
  \hline
  \end{tabular}
  }}
\end{table}

\newpage

\begingroup
\renewcommand{\arraystretch}{1.0}
\setlength{\tabcolsep}{1.8pt}
\fontsize{11pt}{15pt}\selectfont
\begin{longtable}{l|c|c} 
\caption{Comparison of Theorem \ref{ICIbound} and Theorem \ref{Cardosobound} (\cite[Proposition 3.1]{Cardosofilomat}) for some Sagenamed graph and some graph classes. The bounds of Theorem \ref{ICIbound} better than Theorem \ref{Cardosobound} are in bold.}\label{table:injectivechromaticindex}\\
\hline
Graph & Theorem \ref{ICIbound}& Theorem \ref{Cardosobound}\\
\hline
\endfirsthead
\caption[]{Comparison of Theorem \ref{ICIbound} and Theorem \ref{Cardosobound} (\cite[Proposition 3.1]{Cardosofilomat})} \\
\hline
Graph & Theorem \ref{ICIbound}& Theorem \ref{Cardosobound}\\
\hline
\endhead
\hline \multicolumn{3}{r}{Continued on next page} \\
\endfoot
\hline
\endlastfoot
\textbf{Bidiakis Cube}&  4& 1\\     
\textbf{Brouwer-Haemers Graph}& 39& 3\\   
\textbf{Clebsch Graph}&  7& 1 \\    
\textbf{Dodecahedral Graph}&  4 & 1 \\    
\textbf{Frucht Graph}& 4& 3\\     
\textbf{Gosset Graph}&  135&  21\\    
\textbf{Harborth Graph}&  6& 3\\   
\textbf{Heawood Graph} & 3& 1\\    
\textbf{Icosahedral Graph}& 9& 3\\      
\textbf{Klein7-Regular Graph}& 13&  3\\
Krackhardt Kite Graph& 6& 6\\     
\textbf{Meredith Graph}& 5&  1\\     
\textbf{Moser Spindle}& 4& 3\\     
\textbf{Petersen Graph}& 4& 1\\     
\textbf{Shrikhande Graph}& 12& 3\\     
\textbf{Sylvester Graph}&   7&  1\\
\textbf{Tietze Graph}&  4&  3\\     
\textbf{Wells Graph}&   7&  1\\     
\hline
Graph classes& Theorem \ref{ICIbound}& Theorem \ref{Cardosobound}\\
\hline
Book graph $B_n$&$\frac{(2n+1)(n+2\sqrt{n}+5)}{(n+2)(n+2\sqrt{n}+1)}$& 3\\
Complete graph $K_n$& $\frac{n(n-1)}{2}$&$\frac{n(n-1)}{2}$\\
\textbf{Complete bipartite graph $K_{m,n}$}& $\frac{\text{min}\{m,n\}^2+mn}{m+n}$& 1\\
Friendship graph $F_n$& $\frac{3n+6}{2n+1}$&3\\
Star graph $K_{1,n}$& 1& 1\\
\hline
\end{longtable}
\endgroup


\newpage
\begingroup
\renewcommand{\arraystretch}{0.65}
\setlength{\tabcolsep}{1.8pt}
\fontsize{11pt}{15pt}\selectfont
\begin{longtable}{l|c|c|c|c|c}
\caption{Comparison of Conjecture \ref{conjecture} (\cite[Section1]{FSGT1989}), Theorem \ref{SCIbound1},Theorem \ref{SCIbound2}and Corollary \ref{AR2026} (\cite[Corollary 9]{AR2026}) for Sage named $d$-regular graphs. Tight bounds and improvements are indicated in bold.}
\label{table:strongchromaticindex-dregular} \\
\hline
Graph & Conjecture \ref{conjecture} & Theorem \ref{SCIbound1}&Theorem \ref{SCIbound2}& Corollary \ref{AR2026} & $\chi_s'(G)$ \\
\hline
\endfirsthead
\caption[]{Comparison of Conjecture \ref{conjecture}, Theorem \ref{SCIbound1},Theorem \ref{SCIbound2} and Corollary \ref{AR2026} (\cite[Corollary 9]{AR2026}) (continued).} \\
\hline
Graph & Conjecture \ref{conjecture} & Theorem \ref{SCIbound1}& Theorem \ref{SCIbound2} & Corollary \ref{AR2026} & $\chi_s'(G)$ \\
\hline
\endhead
\hline \multicolumn{5}{r}{Continued on next page} \\
\endfoot
\hline
\endlastfoot
Balaban 10-cage& 10& 5& 5&5& 6\\
Balaban 11-cage&10 & 5& 5& 6& 6\\
Bidiakis cube& 10& 5& 5& 6& 8\\
\textbf{Biggs-Smith graph}& 10& \textbf{6}& 5& 6& 6\\
Blanusa First Snark Graph& 10& 4& 5&6& 7\\
Blanusa Second Snark Graph& 10& 4& 5& 6& 7\\
Brinkmann graph& 20& 5& 8& 9& 10\\
Brouwer-Haemers& 500& 41& 68& 68& -\\
\textbf{Bucky Ball}& 10& 4& \textbf{5}&5& 5\\
Cell 600& 180& 24& 41& 43& -\\
Chvatal graph& 20& 6& 7&8& 12\\
\textbf{Clebsch graph}& 29& 8& \textbf{10}&10& 10\\
Conway-Smith graph for 3S7& 125& 12& 28& 29& -\\
Coxeter Graph& 10& 5& 5&6& 7\\
\textbf{Desargues Graph}& 10& 3& \textbf{5}&5& 5\\
Dejter Graph& 45& 6& 11&12& -\\
\textbf{Dodecahedron}& 10&\textbf{5}& \textbf{5}&5& 5\\
Double star snark& 10& 5& 5&5& 7\\
Durer graph& 10& 4& 5&6& 6\\
Dyck graph& 10& 3& 5&6& 6\\
Ellingham-Horton 54-graph& 10& 4& 5&6& 6\\
Ellingham-Horton 78-graph& 10& 4& 5&6& 6\\
F26A Graph& 10& 6& 5&6& 7\\
Flower Snark& 10& 4& 5&5& 6\\
Folkman Graph& 20& 8& 7&8& 10\\
Foster Graph& 10& 4& 5&5& 6\\
Foster graph for 3.Sym(6) graph& 45& 7& 14&14& 15\\
Franklin graph& 10& 3& 5&6& 6\\
Frucht graph& 10& 5& 5&6& 6\\
Goldner-Harary graph& 10& 7& 5&10& 24\\
Golomb graph& 10& 5& 5&7& 11\\
Gray graph& 10& 5& 5&6& 6\\
\textbf{Grotzsch graph}& 20& 7& \textbf{8}&7& 10\\
Harborth Graph& 20& 7& 8&8& 9\\
Harries Graph& 10& 5& 5&6& -\\
Harries-Wong graph& 10& 5& 5&6& 6\\
Heawood graph& 10& 3& 5&7& 7\\
Herschel graph& 10& 5& 5&6& 9\\
Hexahedron& 10& 3& 5&6& 6\\
Hoffman Graph& 20& 7& 7&8& 12\\
Hoffman-Singleton graph& 58& 9& 18&18& -\\
Holt graph& 20& 6& 8&8& 9\\
Horton Graph& 10& 4& 5&6& 6\\
Icosahedron& 29& 8& 12&15& 15\\
Klein 3-regular Graph& 10& 5& 5&6& 7\\
Klein 7-regular Graph& 58& 10& 19&21& 21\\
\textbf{Krackhardt Kite Graph}& 20& 6& \textbf{8}&7& 14\\
Ljubljana graph& 10& 5& 5&6& 6\\
M22 Graph& 320& 30& 51&52& -\\
\textbf{Markstroem Graph}& 10& \textbf{6}& 5&6& 6\\
McGee graph& 10& 5& 5&6& 7\\
Meredith Graph& 20& 7& 7&7& 13\\
Moebius-Kantor Graph& 10& 3& 5&6& 6\\
Moser spindle& 10& 6& 6&6& 9\\
\textbf{Murty Graph}& 20& \textbf{7}& \textbf{8}&6& 10\\
Nauru Graph& 10& 4& 5&6& 6\\
Pappus Graph& 10& 4& 5& 6& 6\\
Perkel Graph& 45& 9& 14&15& -\\
\textbf{Petersen graph}& 10& 4&\textbf{5}&5& 5\\
\textbf{Poussin Graph}& 29& 10& \textbf{11}&10& 17\\
Robertson Graph& 20& 5& 8&8& 10\\
\textbf{Shrikhande graph}& 45& 7& \textbf{16}& 16& 16\\
Sims-Gewirtz Graph& 125& 14& 28&28& -\\
\textbf{Sousselier Graph}& 20& 5&\textbf{7}&6& 7\\
\textbf{Sylvester Graph}& 29& 6& \textbf{10}&10& 10\\
\textbf{Szekeres Snark Graph}& 10& 4& \textbf{5}&5& 5\\
Tietze Graph& 10& 3& 5&6& 7\\
Tricorn Graph& 10& 4& 5&5& 7\\
Truncated Tetrahedron& 10& 5& 5&6& 6\\
Tutte 12-Cage& 10& 4& 5&6& 6\\
Tutte-Coxeter graph& 10 & 5& 5&5& 7\\
Twinplex Graph& 10& 4& 5&6& 7\\
Wagner Graph& 10& 4& 5&6& 10\\
\textbf{Wells graph}& 29& 7& \textbf{10}&10& 10\\
\textbf{Wiener-Araya Graph}& 20& 5& \textbf{8}&5& 8\\
\end{longtable}
\endgroup


\end{document}